\documentclass[11pt,twoside]{amsart}

\usepackage{graphicx}

\nonstopmode

\newtheorem{thm}{Theorem}[section]
\newtheorem{prop}[thm]{Proposition}
\newtheorem{lem}[thm]{Lemma}
\newtheorem{cor}[thm]{Corollary}

\theoremstyle{definition}
\newtheorem{defn}[thm]{Definition}

\newtheorem{example}[thm]{Example}
\newtheorem{rem}[thm]{Remark}

\newtheorem{alg}[thm]{Algorithm}
\newtheorem{question}[thm]{Question}

\theoremstyle{remark}

\newcommand{\mif}{\mbox{if} ~}

\newcommand{\skipit}[1]{{}}
\newcommand{\prfend}{\hbox to7pt{\hfil}
\par\vskip-\baselineskip\hbox to\hsize
{\hfil\vbox {\hrule width6pt height6pt}}\vskip\baselineskip}

\newcommand {\PP}{\mathbb{P}}

\newcommand{\HH}{H_{\mathfrak m}}

\DeclareMathOperator{\reg}{reg}

\newcommand{\myarrow}[2]{\hbox to #1pt{\hfil$\to$\hfil}{\hskip-#1pt{\raise
10pt\hbox to#1pt{\hfil$\scriptscriptstyle #2$\hfil}}}}

\begin{document}

\title{Numerical Macaulification}

\author{Juan Migliore}

\author{Uwe Nagel}

\address{}

\email{}


\date{}


\thanks{This work was partially supported by two grants from the
Simons Foundation (\#208579 to Juan Migliore and \#208869 to Uwe Nagel)}

\begin{abstract}
An unpublished example due to Joe Harris from 1983 (or earlier) gave two smooth space curves with the same Hilbert function, but one of the curves was arithmetically Cohen-Macaulay (ACM) and the other was not.  
Starting with an arbitrary homogeneous ideal in any number of variables, we give two constructions, each of which produces, in a finite number of steps, an ideal with the Hilbert function of a codimension two ACM subscheme.  We call the subscheme associated to such an ideal ``numerically ACM."  We study the connections between these two constructions, and in particular show that they produce ideals with the same Hilbert function.  We call the resulting ideal from either construction  a ``numerical Macaulification'' of the original ideal.  
Specializing to the case where the ideals are unmixed of codimension two, we show that (a) every even liaison class, $\mathcal L$, contains numerically ACM subschemes, (b) the subset, $\mathcal M$, of numerically ACM subschemes in $\mathcal L$ has, by itself, a Lazarsfeld-Rao structure, and (c) the numerical Macaulification of a minimal element of $\mathcal L$  is a minimal element of $\mathcal M$.  Finally, if we further restrict to curves in $\mathbb P^3$, we show that the even liaison class of curves with Hartshorne-Rao module concentrated in one degree and having dimension $n$ contains smooth, numerically ACM curves, for all $n \geq 1$.  The first (and smallest) such example is that of Harris. A consequence of our results is that the knowledge of the Hilbert function of an integral curve alone is not enough to decide whether it contains zero-dimensional arithmetically Gorenstein subschemes of arbitrarily large degree.

\end{abstract}

\maketitle

\begin{center}
{\em Dedicated to Joe Harris on the occasion of his 60th birthday}
\end{center}

\section{Introduction}

A natural, and very old, problem is to determine what information can be obtained about an algebraic variety, $X$, based on knowledge of its Hilbert function, possibly with some reasonable additional assumptions on $X$.  There is a vast literature on this subject.  It is sometimes the case that one can determine whether or not $X$ is arithmetically Cohen-Macaulay (ACM), based on the Hilbert function.  An old example due to Joe Harris \cite{harris} shows a limitation to this, by exhibiting two smooth curves in $\mathbb P^3$ with the same Hilbert function, but one ACM and the other not.  Harris's example consists of two curves of degree 10 and genus 11, where the non-ACM one is in the liaison class of a set of two skew lines.  In Section \ref{harris ex} we recall this example.

The ACM property provides an important distinction between two curves, even irreducible ones with the same Hilbert function.  For instance, if $X$ is a reduced, irreducible ACM curve then it has arithmetically Gorenstein subsets of arbitrarily large degree \cite{KMMNP}.  However, if $X$ is irreducible but not ACM then we will see that there is a bound on the degree of an arithmetically Gorenstein zero-dimensional subscheme that it contains (Proposition \ref{gor subscheme}).  Thus, the question of whether  the number of degrees of arithmetically Gorenstein subschemes of an integral curve is finite or infinite cannot be solved (in general) only from knowing the Hilbert function of the curve.

In this note we will develop a different approach to this problem than that of Harris.  In Lemma \ref{lem:char-num-aCM} we give several equivalent conditions for a standard graded algebra to have the same Hilbert function as that of a Cohen-Macaulay ideal of codimension two. Some are from the perspective of the graded Betti numbers, and some from the perspective of the Hilbert function, especially the $h$-vector.  These lead to  Algorithm \ref{algorithm} in Section \ref{num mac section} that starts with an arbitrary homogeneous ideal and produces, after a finite number of repetitions of a construction called {\em basic double linkage}, an ideal with the Hilbert function of an ACM codimension two subscheme.  We call such an ideal {\em numerically ACM}. For a later application we also provide an alternative algorithm (cf.\ Algorithm \ref{alg2}) that also takes an arbitrary ideal to one with the Hilbert function of an ACM codimension two subscheme. In fact, we show that its resulting ideal has the same Hilbert function as the ideal obtained from the original ideal via Algorithm \ref{algorithm} (see Proposition \ref{prop:relate-alg}) though it typically has more minimal generators than the latter. We give examples in Section \ref{example section}.

In the event that our original ideal is the ideal of an unmixed codimension two subscheme of $\mathbb P^n$, we show something more: if $\mathcal L$ is an even liaison class of codimension two subschemes of $\mathbb P^n$ then the subset, $\mathcal M$, of numerically ACM  schemes in $\mathcal L$ satisfies a Lazarsfeld-Rao property much like that satisfied by $\mathcal L$ itself.  This is Theorem \ref{LR for num ACM}.  To achieve this, we combine results about the Lazarsfeld-Rao property in $\mathcal L$
with an analysis of the change of the Hilbert function when carrying out Algorithm \ref{alg2}.

In Section \ref{smooth} we  generalize the example of Harris, following our approach, and show in Theorem \ref{family of smooth NACM} that in any liaison class of space curves corresponding to a Hartshorne-Rao module of diameter one and dimension $n$, there are smooth, maximal rank, numerically ACM curves.  The example of Harris is the first case, $n=1$.


\section{Harris's original example} \label{harris ex}

With his permission, we first give the example, due to Joe Harris, of two smooth curves $C$, $C'$ with the same Hilbert function, one ACM and the other not. We quote directly from \cite{harris}.

The point is, to say the Hilbert function is the same for $C$ and $C'$ is to say the ranks of the maps
\[
\rho_n : H^0(\mathbb P^3,\mathcal O(n)) \rightarrow H^0(C,\mathcal O(n)) \ \ \ \hbox{ and } \ \ \
\rho_n' : H^0(\mathbb P^3,\mathcal O(n)) \rightarrow H^0(C',\mathcal O(n))
\]
are the same for all $n$ (of course the degree and genus of $C$ and $C'$ must therefore be the same); but to say that $C$ is ACM and $C'$ is not means that $\rho_n$ is surjective for all $n$, while $\rho_n'$ is not, for some $n$.  Thus we must have $h^0(C,\mathcal O(n)) \neq h^0(C',\mathcal O(n))$ for some $n$; since it would be simplest if $C,C'$ were linearly normal, we might as well take $n=2$ here.  Finally (just as a matter of personal preference) let's arrange for $C'$ to be semicanonical --- i.e. $K_{C'} = \mathcal O_{C'}(2)$, so that $h^0(\mathcal O_{C'}(2)) = g$ --- while $C$ is not, so $h^0(\mathcal O_C(2)) = g-1$.  If we at the same time assume that $C,C'$ do not lie on  quadric surfaces, this then says that $g=11, d=10$.

For $C'$, take $S$ a quartic containing two skew lines $L_1,L_2$, and let $C'$ be a general member of the linear system $C' \in | \mathcal O_S(2)(L_1+L_2)|.$
Since $L_i^2 = -2$ on $S$, $C' \cdot L_i = 0$, --- i.e. $C'$ is disjoint from $L_i$ --- so that, inasmuch as $K_S = O_S$,
\[
K_{C'} = \mathcal O_{C'}(C') = \mathcal O_{C'}(2)
\]
so $C'$ is semicanonical.  Note that $C'$ lies on no cubic surfaces $T$: if it did, we could write $S \cdot T = C' + D$ and then on $S$ we would have $\mathcal O_S(D+L_1 +L_2) = \mathcal O_S(1)$; but $L_1$ and $L_2$ do not lie in a hyperplane.  Thus the Hilbert function of $C'$ is
\[
\begin{array}{rcl}
h(1) & = & 4 \\
h(2) & = & 10 \\
h(3) & = & 20 \\
h(4) & = & 30 \\
 & \cdots \\
h(n) & = & 10n - 10 \hbox{ \ \ ($=$ Hilbert polynomial)}
\end{array}
\]

You may also recognize $C'$ as being residual, in the intersection of two quartic surfaces $S$ and $T$, to a curve of type $(2,4)$ on a quadric; the Hilbert function and semicanonicality can be deduced from this.

As for $C$, let now $S$ be a quartic surface containing a smooth, non-hyperelliptic curve $D$ of degree 6 and genus 3 in $\mathbb P^3$, and let $C$ be a member of the linear series $C \in |\mathcal O_S(1)(D)|$.  Since $D$ does lie on a cubic surface $U$, and we can write $U \cap S = D+E$ where $E$ is again a septic of genus 3, $C$ can be described as the curve residual to a non-hyperelliptic sextic of genus 3 in a complete intersection of quartics --- i.e. $C \in |\mathcal O_S(4)(-E)|$. $C$ is thus linked to a twisted cubic, and so is ACM; likewise, since $D$ lies on no quadrics we see that the Hilbert function of $C$ is equal to that of $C'$.


\section{background}

We will consider homogeneous ideals in the polynomial ring $R = K[x_0,\ldots,x_n]$ or projective subschemes $X \subset \PP^n := \PP^n$ whose codimension is at least two, where $K$ is an infinite field. Thus,  if $I \subset R$ is an ideal of codimension $c$ then the projective dimension of $R/I$ satisfies
\[
c \leq \hbox{pd}(R/I) \leq n+1
\]
and so $\hbox{pd}(I) \leq n$.

The following result will be a key tool throughout this note (cf., e.g, \cite{MN-MathZ} for a proof).

\begin{lem} \label{bdl res}
Let $I \subset R$ be any non-zero ideal.
Let $F \in I$ be any non-zero element of degree $d$, and let $G \in R$ be a general form of degree $a$.  Consider the ideal $J = G \cdot I + (F)$. Then

\begin{itemize}
\item[(a)] $J$ has codimension two.  \\

\item[(b)] If $I$ has a minimal free resolution
 \[
0 \rightarrow \mathbb F_n \rightarrow \mathbb F_{n-1} \rightarrow \cdots \rightarrow \mathbb F_1 \rightarrow I \rightarrow 0
\]
  then $J$ has a free resolution
  \[
0 \rightarrow \mathbb F_n(-a)  \rightarrow \mathbb F_{n-1} (-a) \rightarrow \cdots \rightarrow \mathbb F_3 (-a)  \rightarrow
\begin{array}{c}
\mathbb F_2 (-a) \\
\oplus \\
R(-d-a)
\end{array}
\rightarrow
\begin{array}{c}
\mathbb F_1 (-a) \\
\oplus \\
R(-d)
\end{array}
 \rightarrow J\rightarrow 0.
\]
The only possible cancelation is a copy of $R(-a-d)$ between the first and second free modules, which occurs if and only if there is a minimal generating set for $I$ that includes~$F$. \\

\item[(c)] Let $\mathfrak a = \langle F,G \rangle$.  We have the formula for the Hilbert functions:
\[
h_{R/J}(t) = h_{R/I}(t-a) + h_{R/\mathfrak a}(t).
\]

\item[(d)] There are graded isomorphisms between local cohomology modules with support in the maximal ideal ${\mathfrak m} = (x_0,\ldots,x_n)$
\begin{equation*}
  \HH^i (R/J) \cong \HH^i (R/I)(-a)
\end{equation*}
whenever $i \le n-2$. \\

\item[(e)] If we have $I = I_X$, the saturated ideal of a subscheme $X \subset \mathbb P^n$, then  $J$ is also a saturated ideal, defining a codimension two subscheme $X_1$. \\

\item[(f)] If $X$ is unmixed of  codimension two, then the degree of $X_1$ is $a d  + \deg X$, and, as sets, $X_1$ is the union of $X$ and the complete intersection of $F$ and $G$.
Furthermore, in this case
$X$ is linked to $X_1$ in two steps.
\end{itemize}

\end{lem}

\begin{defn} \label{def of bdl}
The ideal $J$ produced in Lemma \ref{bdl res} will be called a {\em basic double link of $I$ of type $(d,a)$}.  We also sometimes call $a$ the {\em height} of the basic double link.
\end{defn}

\noindent Notice that if $X$ is not unmixed of codimension two then the construction given in Lemma \ref{bdl res} is not actually related to linkage, but we retain the terminology since it is the standard one.  We also note that basic double linkage has been generalized, in the context of Gorenstein liaison, to a construction called {\em basic double G-linkage} -- cf. \cite{KMMNP}.

Consider a free resolution
\[
0 \rightarrow \mathbb F_n \rightarrow \cdots \rightarrow \mathbb F_1 \rightarrow R \rightarrow R/I \rightarrow 0.
\]
Let $D$ be the associated Betti diagram.
It is well-known that the Hilbert function of $R/I$ can be computed from $D$, and that any free summand $R(-a)$ that occurs in consecutive $\mathbb F_i$ does not contribute to the Hilbert function computation, and so can be canceled from the numerical computation, even if removing it produces a diagram that is not the Betti diagram of any module.
Several papers have used the idea of formally canceling free summands in this way, in order to obtain useful numerical information (e.g.\ \cite{GL}, \cite{peeva}, \cite{GHMS}, \cite{collect}).  We make a small extension of this idea by applying it to a diagram that may or may not be a Betti diagram.

\begin{defn}
Consider a finite diagram, $D = (d_{i,j})$, with $0 \leq j \leq n$, $0 \leq i$, $d_{0,0} = 1$, and $d_{i,0} = 0$ for $i \geq 2$.  The {\em numerical reduction} of $D$ is the diagram obtained as follows.  Whenever $d_{i,j} > 0$ and $d_{i-1,j+1} > 0$, let $d = \min\{d_{i,j}, d_{i-1,j+1}\}$.  Then replace $d_{i,j}$ by $d_{i,j} - d$ and replace $d_{i-1,j+1}$ by $d_{i-1,j+1}-d$.
Two ideals, $I_1$ and $I_2$, are {\em numerically equivalent} if their Betti diagrams have the same numerical reduction.
\end{defn}

\noindent Clearly two numerically equivalent ideals have the same Hilbert function.


%

\begin{rem} \label{sauer rem}
Fix integers
\[
s_1 \geq s_2 \geq \dots \geq s_{\nu-1} > 0, \ \ \ \   r_1 \geq r_2 \geq \dots \geq r_\nu > 0
\]
such that $\sum_{j=1}^{\nu} r_j = \sum_{i=1}^{\nu-1} s_i$.  Consider the $(\nu-1)\times \nu$ integer matrix
\[
A = (s_i -r_j).
\]
Notice that the columns are non-decreasing from bottom to top, and the rows are non-decreasing from left to right.  Sauer \cite{sauer} remarks without proof (page 84) that there is an ACM curve $Y \subset \mathbb P^3$ with free resolution
\begin{equation} \label{resol}
0 \rightarrow \bigoplus_{i=1}^{\nu-1} R(-s_i) \rightarrow \bigoplus_{j=1}^\nu R(-r_j) \rightarrow I_Y \rightarrow 0
\end{equation}
if and only if the entries of the main diagonal of $A$ are non-negative.  We make some additional remarks.

\begin{enumerate}
\item If we restrict to minimal free resolutions, then we have the stronger condition that the entries of the main diagonal are strictly positive.

\item If we do not restrict to minimal free resolutions, then the criterion of Sauer is not correct.  For instance, choosing $s_1 = 4$ and $r_1 = r_2 = 2$ clearly gives the Koszul resolution for a complete intersection of type $(2,2)$, but adding a trivial summand $R(-1)$ to both free modules produces a non-minimal free resolution with a negative entry on the main diagonal.
\end{enumerate}
\end{rem}

We will adjust Sauer's criterion. It works equally  well for ACM codimension two subvarieties of $\mathbb P^n$. Before stating the result, let us introduce some notation. We denote the Hilbert function of $R/I$ by $h_{R/I} (j) = \dim_K [R/I]_j$. If $I$ has codimension two, then its Hilbert series can be written as
\begin{equation*}
  \sum_{j \ge 0} h_{R/I} (j) z^j = \frac{h_0 + h_1 z + \cdots + h_e z^e}{(1-z)^{n-1}},
\end{equation*}
where $h_e \neq 0$.
Then $h = (h_0,\ldots,h_e)$ is called the  $h$-vector of $R/I$. Equivalently, the $h$-vector is the list of non-zero values of the $(n-1)$st difference of the Hilbert function, $\Delta^{n-1} h_{R/I}$, where $\Delta h_{R/I} (j) = h_{R/I} (j) - h_{R/I} (j-1)$ and $\Delta^{i} h_{R/I} = \Delta (\Delta^{i-1} h_{R/I})$ if $i \ge 1$. Abusing terminology, in this note we define the {\em $h$-vector} of $R/I$ to be always the non-zero values of $\Delta^{n-1} h_{R/I}$, even if $I$ does not have codimension two.
Then we have:

\begin{lem}
  \label{lem:char-num-aCM}
Let $I \subset R$ be a homogeneous ideal whose codimension is at least two. Consider a (not necessarily minimal) free resolution of $R/I$:
\[
0 \rightarrow \bigoplus_{k=1}^n(-c_{n,k}) \rightarrow \bigoplus_{k=1}^n(-c_{n-1,k}) \rightarrow  \cdots \rightarrow \bigoplus_{k=1}^n(-c_{2,k}) \rightarrow \bigoplus_{k=1}^n(-c_{1,k}) \rightarrow R \rightarrow R/I \rightarrow 0.
\]
Let
\[
\underline s = \{ s_i \} = \{ c_{p,q} \  | \  p \hbox{ is even} \}, \ \ \ \underline r = \{ r_j \} = \{ c_{p,q} \ | \ p \hbox{ is odd} \},
\]
and assume that the sets $\underline s$ and $\underline r$ are ordered so that
\[
s_1 \geq s_2 \geq \dots \geq s_{\nu-1}, \ \ \ \   r_1 \geq r_2 \geq \dots \geq r_\nu.
\]
Note that
\[
\sum_{j=1}^{\nu} r_j = \sum_{i=1}^{\nu-1} s_i.
\]

Assume that $s_{\nu - 1} > r_{\nu} \ge 1$. Then the following conditions are equivalent:
\begin{itemize}
  \item[(a)] $I$ has the same Hilbert function as that of a CM ideal of codimension two.

  \item[(b)] The $h$-vector of $R/I$ is an O-sequence.

  \item[(c)] For every integer $k \ge r_{\nu}$,
  \[
  \#\{r_i \ | \ r_i \le k\} > \#\{s_j \ | \ s_j \le k\}.
  \]

  \item[(d)] For all $i = 1,\ldots,\nu - 1$, \; $s_i \geq r_i$.

  \item [(e)] There is a (CM) ideal $J \subset R$ of codimension two having a free resolution of the form
\begin{equation} \label{seq:CM-resol}
0 \rightarrow \bigoplus_{i=1}^{\nu-1} R(-s_i) \rightarrow \bigoplus_{j=1}^\nu R(-r_j) \rightarrow J \rightarrow 0
\end{equation}
\end{itemize}
\end{lem}

As preparation for its proof we need:

\begin{lem}
  \label{lem:h-vector-computtion}
Using the notation of the preceding lemma, the $h$-vector of $R/I$ can be computed using, for each $k$,  the formula
\begin{equation}
  \label{eq:h-vector}
  h(k) = k+1 - \sum_{r_i \le k} (k - r_i + 1) + \sum_{s_j \le k} (k - s_j + 1).
\end{equation}
It follows that
\begin{equation}
  \label{eq:delta-h}
  h (k+1) = h(k) + 1 -  \#\{r_i \ | \ r_i \le k+1\} +  \#\{s_j \ | \ s_j \le k+1\}.
\end{equation}
\end{lem}

\begin{proof}
The additivity of vector space dimension along exact sequences provides for each integer $k$
\begin{equation*}
   h_{R/I} (k) = h_{R}(k) - \sum_{i} h_{R} (k - r_i) + \sum_{j} h_{R} (k - s_j).
\end{equation*}
The claim follows by passing to the $(n-1)$st differences as $\Delta^{n-1} h_{R} (k) = k+1$ if $k \ge 0$.
\end{proof}

Now we are ready to come back to Lemma \ref{lem:char-num-aCM}.

\begin{proof}[Proof of Lemma \ref{lem:char-num-aCM}]
This is probably known to specialists. However, for the convenience of the reader we  provide a brief argument.

If $R/J$ is Cohen-Macaulay of dimension $n-1$, then its $h$-vector is the Hilbert function of $R/(I,\ell_1,\ldots,\ell_{n-1})$, where $\ell_1,\ldots,\ell_{n-1} \in R$ are general linear forms. Thus, (b) is a consequence of (a).

Lemma \ref{lem:h-vector-computtion} shows that (c) follows from (b) because the $h$-vector is weakly decreasing after it stopped to increase strictly.

It is elementary to check that (c) provides (d).

Assume Condition (d) is true. Then let $A = (a_{i, j})$ be the $(\nu - 1) \times \nu$ matrix such that $a_{i, i} = x_0^{s_i - r_i}, \; a_{i, i+1} = x_1^{s_i - r_{i+1}}$, and $a_{i, j} = 0$ if $j \neq i, i+1$. Then the ideal $J$ generated by the maximal minors of $A$ has codimension two and a free resolution of the form \eqref{seq:CM-resol}, establishing Condition (e).

Condition (a) follows from (e) as $I$ and $J$ have the same Hilbert function.
\end{proof}

We need the following consequence.

\begin{cor}
\begin{itemize}
  \item[(a)] Fix integers $s_1 \geq s_2 \geq \dots \geq s_{\nu-1} > 0$ and  $ r_1 \geq r_2 \geq \dots \geq r_\nu > 0$
such that $\sum_{j=1}^{\nu} r_j = \sum_{i=1}^{\nu-1} s_i$.  Consider the $(\nu-1)\times \nu$ integer matrix $A = (s_i -r_j)$.  Then there is a codimension two subscheme $Y \subset \mathbb P^n$ with minimal free resolution \eqref{resol} if and only if the entries of the main diagonal of $A$ are strictly positive.

  \item[(b)] Let $\underline s$ and $\underline r$ be sets satisfying the conditions in (a).  Suppose that equal entries  are added to both sets, all $ \geq r_\nu$  (corresponding to the addition of trivial summands $R( \ )$ to both free modules in the resolution).  Let $f(t) = \# \{ i \ | \ s_i \leq t \}$ and $g(t) = \# \{j \ | \  r_j  \leq t \}$.  Then for all $t$ we have $g(t) > f(t)$.
\end{itemize}

\end{cor}


\section{Numerical Macaulification: Two Algorithms} \label{num mac section}

In this section we will introduce some terminology used throughout the paper.  The main goal of the section, though, is to give two algorithms to produce, from an arbitrary ideal $I$ of height $\geq 2$, a numerically ACM ideal, using only basic double linkage.  The Hilbert functions of the two resulting ideals turn out to be equal.  Thus we will call the end result of these constructions the {\em numerical Macaulification} of $I$.  The result is numerically unique but not unique as an ideal, since there are two algorithms, and even within one algorithm there are several choices of polynomials (of fixed degree).

\begin{defn}
A homogeneous ideal $J \subset R$  is {\em numerically $r$-ACM} if $R/J$ has the Hilbert function of some codimension $r$ ACM subscheme of $\mathbb P^n$.  When $r=2$ we will simply say that $J$ is {\em numerically ACM}.
\end{defn}

The main result of this section is the following.

\begin{thm}
If I is an ideal whose codimension is at least two, then there is a finite sequence of basic double links, starting from $I$, that results in an ideal that is numerically ACM.
\end{thm}

To achieve this, we give two algorithms.  Then we will compare the algorithms to see that they result in ideals with the same Hilbert function.

\begin{alg} \label{algorithm}
Let $I \subset R$ be a homogeneous ideal of height $\geq 2$.  Consider a minimal free resolution of $R/I$:
\[
0 \rightarrow \bigoplus_{k=1}^n(-c_{n,k}) \rightarrow \bigoplus_{k=1}^n(-c_{n-1,k}) \rightarrow  \cdots \rightarrow \bigoplus_{k=1}^n(-c_{2,k}) \rightarrow \bigoplus_{k=1}^n(-c_{1,k}) \rightarrow R \rightarrow R/I \rightarrow 0.
\]
\begin{enumerate}
 \item Let
\[
\underline s = \{ s_i \} = \{ c_{p,q} \  | \  p \hbox{ is even} \}, \ \ \ \underline r = \{ r_j \} = \{ c_{p,q} \ | \ p \hbox{ is odd} \},
\]
and assume that the sets $\underline s$ and $\underline r$ are ordered so that the entries are non-increasing.  Note that $\sum s_i = \sum r_j$.

\item Remove equal elements $r_j$ and $s_i$ pairwise one at a time.  For convenience of notation, we will still call the sets $\underline s$ and $\underline r$. So now we may assume that $\underline{s}$ and $\underline{r}$ are disjoint sets.

\item Form the matrix $A = (s_i - r_j)$.  Let $\{ -d_1, \dots, - d_\ell \}$ be the negative entries of $A$ on the main diagonal.  Assume for convenience that they are ordered according to  {\em non-decreasing} values of $r_j$.  (That is, we are taking the negative entries of the main diagonal beginning from the bottom right and moving up and left, regardless of the values of the $d_k$.)

\item {\bf (Main step)}  Say that $-d_1 = s_{i_1} - r_{i_1} <0$ (since it is on the main diagonal).  Using general polynomials, let $J$ be the ideal obtained from $I$ by a basic double link of type $(r_{i_1},d_1)$.

\item Repeat steps (1) -- (4) for $J$.  Continue repeating until there are no longer negative entries on the main diagonal.

\end{enumerate}

\end{alg}

\begin{prop}
  \label{prop:corr-first-alg}
This algorithm terminates, and the result is an ideal  that is numerically ACM.
Thus, we define the resulting ideal to be the {\em numerical Macaulification} of $I$.
\end{prop}

\begin{proof}

Observe first that thanks to Lemma \ref{bdl res}, the Betti numbers (up to one possible cancelation) and hence the Hilbert function of the resulting ideal depend only on the degrees of the polynomials used.  Let $I$ be the original ideal,  and $J$ the result of performing steps (1) to (4). As a result of step (2), associated to $I$ are the sets $\underline s = \{ s_i \}$ and $\underline r = \{ r_j \}$, with no common entries. Then $J$ has a (not necessarily minimal) free resolution with Betti numbers giving new lists
\[
\underline s' = \{ s_i + d_1 \} \ \cup \ \{ r_{i_1} + d_1 \}, \ \ \  \underline r' = \{ r_j +  d_1 \} \ \cup \ \{ r_{i_1} \}.
\]
(Note that $\{s_i + d_1\}$ and $\{r_j + d_1\}$ contain, in general, more than one element, while $\{r_{i_1} + d_1\}$ and $\{ r_{i_1} \}$ are sets with one single element.)  Thus $\underline s'$ contains at least one $  s_{i_1} + d_1 = r_{i_1}$ and at least  one $r_{i_1} + d_1$, and $\underline r'$ contains at least one $r_{i_1} +d_1$ and at least one $r_{i_1}$.  In performing step (2) for $J$, we remove these two entries from both lists.  One checks that as the result of this removal, the new matrix $A$ is obtained from the original one by removing row $i_1$ and column $i_1$.  Thus the new matrix has the same entries on the main diagonal as the original one, except that one negative entry (namely $-d_1$) has been removed.  Thus the algorithm terminates.  But one result of the algorithm are two lists, $\underline s'$ and $\underline r'$, satisfying the conditions of Lemma \ref{lem:char-num-aCM}. (Notice that the construction guarantees that $A$ will have no entries that are equal to 0, in particular on the main diagonal.)  Since the Hilbert function of $R/J$ (where $J$ is the result of the completion of the algorithm) can be computed from $\underline r'$ and $\underline s'$ and seen to be the same as that of the ACM subscheme determined by $\underline r'$ and $\underline s'$, the result follows from Lemma \ref{lem:char-num-aCM}.
\end{proof}

\begin{example} \label{ci33}
Let $I = (w^3,x^3) \cap (y^3,z^3) \subset k[w,x,y,z]$.
This curve has Betti diagram
\begin{verbatim}
                                  0    1    2    3
                          -------------------------
                           0:     1    -    -    -
                           1:     -    -    -    -
                           2:     -    -    -    -
                           3:     -    -    -    -
                           4:     -    -    -    -
                           5:     -    4    -    -
                           6:     -    -    -    -
                           7:     -    -    4    -
                           8:     -    -    -    -
                           9:     -    -    -    1
                          -------------------------
                          Tot:    1    4    4    1
\end{verbatim}
and $h$-vector $(1, 2, 3, 4, 5, 6, 3, 0, -3, -2, -1)$.

The two lists $\underline s$ and $\underline r$ are
\[
\underline s = \{ 9,9,9,9 \}, \ \  \ \ \underline r = \{ 12, 6,6,6,6 \}
\]
so the matrix has the form
\begin{equation} \label{ci33 matrix}
A =
\left [
\begin{array}{ccccccccccccccccc}
\fbox{$-3$} & 3 & 3 & 3 & 3 \\
-3 & 3 & 3 & 3 & 3 \\
-3 & 3 & 3 & 3 & 3 \\
-3 & 3 & 3 & 3 & 3
\end{array}
\right ]
\end{equation}
with only one negative entry in the main diagonal.  Thus we perform only one basic double link, using $\deg G = 3$ and $\deg F = 12$.  The $h$-vector of the resulting ideal, $J$, is computed using Lemma~\ref{bdl res}:
\[
\begin{array}{ccccccccccccccccccccc}
&&&1& 2& 3& 4& 5& 6& 3& 0& -3& -2& -1 \\
1 & 2 & 3 & 3 & 3 & 3 & 3 & 3 & 3 & 3 & 3 & 3 & 2 & 1 \\ \hline
1& 2& 3& 4& 5& 6& 7& 8& 9& 6& 3

\end{array}
\]
and Betti diagram
\begin{verbatim}
                                   0    1    2    3
                           -------------------------
                            0:     1    -    -    -
                            1:     -    -    -    -
                                   ...
                            7:     -    -    -    -
                            8:     -    4    -    -
                            9:     -    -    -    -
                           10:     -    -    4    -
                           11:     -    1    -    -
                           12:     -    -    -    1
                           13:     -    -    1    -
                           -------------------------
                           Tot:    1    5    5    1
\end{verbatim}
\end{example}

\begin{alg} \label{alg2}
Let $I \subset R = K[x_0,\dots,x_n]$ be a homogeneous ideal of height $\geq 2$.  Assume for convenience that $I$ contains no linear forms (otherwise use a smaller $R$).  Let $h_{R/I}$ be the Hilbert function of $R/I$, and consider the $(n-2)$-nd difference
\[
\underline{a} = \Delta^{n-2} h_{R/I} = (1,a_1,\dots,a_e)
\]
where $a_e$ is the last non-zero value and $a_1 \le 2$.
Note that this is a finite sequence, and if $I$ has codimension two then this is the $h$-vector of $R/I$.
Repeat the following step until $\underline{a}$ becomes  is an O-sequence:
\[
(*) \ \ \ \ \ \ \
\begin{array}{lll}
\hbox{\em If $\underline{a}$ is not an $O$-sequence then set $a_{e+1} = 0$ and let $i$ be the smallest index such that } \\
\hbox{\em $a_i \le i$   and  $a_i < a_{i+1}$. Let $F \in I$ be a form of degree $i+2$, and let $J = L \cdot I + (F)$,} \\
\hbox{\em with $L$  a general linear form.  Set $\underline{b}$  to be the $(n-2)$-nd difference of $h_{R/J}$.}
\end{array}
\]
\end{alg}

Note that again this algorithm uses basic double links, but this time always of type $(d,1)$ for different $d$. We illustrate its idea.

\begin{example} \label{motivating example}
Let $C$ be the general union of a line $C_1$, a plane cubic $C_2$, and a curve $C_3$ that is linked to a line in a complete intersection of type $(4,8)$.  The Betti diagram for $R/I_C$ has the form

\begin{verbatim}
                                 0    1    2    3
                         -------------------------
                          0:     1    -    -    -
                          1:     -    -    -    -
                          2:     -    -    -    -
                          3:     -    -    -    -
                          4:     -    -    -    -
                          5:     -    2    1    -
                          6:     -    -    -    -
                          7:     -    2    3    1
                          8:     -    -    -    -
                          9:     -    2    1    -
                         10:     -    1    -    -
                         11:     -    1    5    2
                         12:     -    1    3    2
                         -------------------------
                         Tot:    1    9   13    5
\end{verbatim}

\noindent The two lists $\underline s$ and $\underline r$ are
\[
 \underline s = \{ 14, 14, 14, 13, 13, 13, 13, 13, 11, 9, 9, 9, 7 \}, \ \
\underline r = \{ 15, 15, 14, 14, 13, 12, 11, 10, 10, 10, 8, 8, 6, 6 \}.
\]
Thus, the  $h$-vector of the original curve is
\begin{verbatim}
[1, 2, 3, 4, 5, 6, 5, 5, 3, 4, 2, 0, -3, -2]
\end{verbatim}
Following Algorithm \ref{alg2}, we perform four basic double links, using, successively, forms $F_1,\ldots,F_4$ of degree $\deg F_1 = 10, \deg F_2 = 14+1 = 15, \deg F_3 = 15 + 1 + 1 = 17$, and $\deg F_4 = 15 + 1+1+1 = 18$.  The $h$-vectors of the successive basic double links are
\begin{verbatim}
[1, 2, 3, 4, 5, 6, 7, 6, 6, 4, 4, 2, 0, -3, -2]
-------------------------------
[1, 2, 3, 4, 5, 6, 7, 8, 7, 7, 5, 5, 3, 1, -2, -2]
-------------------------------
[1, 2, 3, 4, 5, 6, 7, 8, 9, 8, 8, 6, 6, 4, 2, -1, -1]
-------------------------------
[1, 2, 3, 4, 5, 6, 7, 8, 9, 10, 9, 9, 7, 7, 5, 3]
\end{verbatim}
whereby we recognize that only the last has the Hilbert function of an ACM curve.  We remark that each of these $h$-vectors is obtained by shifting the previous $h$-vector by one and adding a vector consisting of $(\deg F_i)$ 1's.  For example, the first  $h$-vector above is obtained by
\[
\begin{array}{ccccccccccccccccccccc}
 & 1 & 2 & 3 & 4 & 5 & 6 & 5 & 5 & 3 & 4 & 2 & 0 & -3 & -2 \\
1 & 1 & 1 & 1 & 1 & 1 & 1 & 1 & 1 & 1 \\ \hline
1 & 2 & 3 & 4 & 5 & 6 & 7 & 6 & 6 & 4 & 4 & 2 & 0 & -3 & -2
\end{array}
\]
The effect of this first basic double link is to ``fix" the growth from 3 to 4 (from degree 8 to degree 9) in the original $h$-vector, which violates maximal growth of the Hilbert function. This illustrates the idea of the proof of the next result; indeed, a basic double link using a form of smaller degree would be ``wasted," since it would not serve to change any part of the $h$-vector that fails to be an $O$-sequence, and any basic double link using a form of larger degree would forever eliminate our ability to ``fix" the impossible growth from 3 to 4, since subsequent basic double links use forms of strictly larger degrees. Thus in order to obtain a numerically ACM curve in the fewest possible steps, the first basic double link {\em must} be using a form of degree 10.  The other three basic double links are similarly forced. For purposes of comparing with Algorithm \ref{algorithm}, in Example \ref{2nd half of ex 5.9} we give the Betti table of the numerically ACM curve thus obtained.
\end{example}

\begin{prop}
   \label{prop:corr-second-alg}
Algorithm \ref{alg2} terminates, and the result is an ideal $J'$ that is numerically ACM.  Furthermore, the degrees of the forms $F$ used in the repeated applications of step (*) are strictly increasing.
\end{prop}

\begin{proof}
\noindent Observe that if we apply step (*),  $\underline{b}$ is obtained by the computation
\[
\begin{array}{ccccccccccccccccccccccccc}
& 1 & 2 & a_2 & a_3 & \dots & a_i & a_{i+1} & \dots & a_e \\
1 & 1 & 1 & 1 & 1 & \dots & 1 \\ \hline
1 & 2 & 3 & a_2+1 & a_3+1 & \dots & a_i +1 & a_{i+1} & \dots & a_e
\end{array}
\]
By assumption, this is an $O$-sequence up to and including degree $i+1$ (where the value is $b_{i+1} = a_i + 1$), and the possible failure to be an $O$-sequence from (now) degree $i+1$ to degree $i+2$ has decreased by one. If it is still not an $O$-sequence in this degree (because $b_{i+1} = a_i + 1 < a_{i+1} = b_{i+2}$), for the next step we choose $F$ of degree $(i+1)+2 = i+3$, so the degrees of the forms are strictly increasing. If $a_i + 1 = a_{i+1}$ then $\deg F$ increases even more in the next application of step (*).

As long as $a_{i+1} < i+2$, we shall define the {\em deficit in degree $i$} to be $\delta _i = \max \{ 0, a_{i+1} - a_i \}$ and the {\em deficit} to be $\delta = \sum \delta_i$.  An algebra is numerically ACM if and only if its deficit is zero.  Clearly each application of step (*) reduces the deficit by one, hence the result follows.
\end{proof}

We now compare the results of Algorithms \ref{algorithm} and \ref{alg2} when they are applied to the same ideal.

\begin{prop}
  \label{prop:relate-alg}
Let $I$ be a homogeneous ideal that is not Cohen-Macaulay and whose codimension is at least two. Then the numerical Macaulification $J$  produced from $I$ by Algorithm \ref{algorithm} has the same Hilbert function as the ideal $J'$ produced from $I$ by Algorithm \ref{alg2}.  Furthermore, $\sum d_i = \delta$, where $\delta$ is the number of basic double links applied in that algorithm and
\begin{equation*}
  \HH^i (R/J) \cong \HH^i (R/J')
\end{equation*}
whenever $i \le n-2$.
\end{prop}

\begin{proof}
In order to show the claim about the Hilbert function it is enough to prove that the two ideals have the same $(n-2)$-nd differences of their Hilbert functions.  Since all ideals involved except possibly $I$ have codimension two, we will refer to these $(n-2)$-nd differences as $h$-vectors.
Let $\underline s = \{s_1 \ge \dots \ge s_{\nu - 1} \}$  and  $\underline r = \{ r_1 \ge \dots \ge r_{\nu} \}$ be the sequences of integers obtained after applying Steps (1) and (2) of Algorithm \ref{algorithm}. In the proof of Proposition \ref{prop:corr-second-alg} we have seen  that the smallest number of basic double links of type $(t, 1)$ that can be used to obtain a numerically Cohen-Macaulay ideal  starting with $I$ is the deficit
\begin{equation*}
\delta = \sum_{i \ge r_{\nu}} \max \{0, h_{i+1} - h_i \}
\end{equation*}
and that Algorithm \ref{alg2} uses exactly $\delta$ such basic double links. Note that $r_{\nu}$ is the least index $i$ such $h_i \le i$.  Moreover, we know  the $h$-vector of the ideal $J'$ obtained by Algorithm \ref{alg2}. Indeed, denote by $(1,h_1,\ldots,h_e)$ the $h$-vector of $R/I$, and let $m$ be the smallest index $i$ such that $i \ge r_{\nu}$ and $h_i < h_{i+1}$. Then the first entries of the $h$-vector $(1,h'_1,\ldots,h'_{e+\delta})$ of $R/J'$ are given by
\begin{equation*}
  h'_i = \begin{cases}
    i+1 & \mif  i < \delta  \\
    h_{i - \delta} + \delta & \mif  \delta \le i \le m + \delta
  \end{cases}
\end{equation*}
Suppose that the original $h$-vector failed to be an $O$-sequence in another degree. Denote by $m'$ the smallest index $i$ such that $i > m$ and $h_i < h_{i+1}$. Then the next entries of the $h$-vector of $R/J'$ are given by
\begin{equation*}
  h'_i = h_{i-d} + \delta' \quad {\rm whenever} \quad m + \delta < i \le m' + \delta,
\end{equation*}
where
\begin{equation*}
  \delta' = \sum_{i > m} \max \{0, h_{i+1} - h_i \}.
\end{equation*}
Continuing in this fashion we get the $h$-vector of $R/J'$.

We now analyze  Algorithm \ref{algorithm}. To simplify notation, let $k$ be the largest index such that $r_k > s_k$, and set $d = d_1 = r_k - s_k$. Then the main step of Algorithm \ref{algorithm} says that we should perform a basic double link of type $(r_k, d)$. Denote the resulting ideal by $\widetilde{J}$. We want to compare the $h$-vectors of $R/I$ and $R/\widetilde{J}$.

By the choice of $k$, we know that $r_k > s_k \ge s_{k-1} > r_{k-1}$. This implies that whenever $s_k  \le j \le r_k -1$
\begin{equation*}
  \#\{r_i \ | \ r_i \le j \} = \nu - k
\end{equation*}\
and
\begin{equation*}
  \#\{s_i \ | \ s_i \le j \} \ge \nu - k.
\end{equation*}

Hence Lemma \ref{lem:h-vector-computtion} provides
\begin{equation} \label{eq:deficit-I}
  h_{s_k -1} < h_{s_k} < \cdots < h_{r_k - 1}
\end{equation}
and that $m = s_k - 1$ is the smallest index $i$ such that $i \ge r_1$ and $h_i < h_{i+1}$.

Since $d = r_k - s_k < r_k$, a complete intersection of type $(r_k, d)$ has $h$-vector
\begin{equation*}
  (1,2,\ldots,d,d,\ldots,d, d-1,\ldots,2,1),
\end{equation*}
where the last entry is in degree $d+r_k -2$. Hence, by Lemma \ref{bdl res} the $h$-vector $(1,a_1,\ldots,a_s)$ of $R/\widetilde{J}$ satisfies
\begin{equation}
  \label{eq:deficit-J}
a_{r_k - 1 + j} = h_{s_k - 1 + j} + \max \{0, d-j \}
\end{equation}
whenever $j \ge 0$.
This means, in particular, that $h_{s_k -1}$ is increased by $d$, $h_{s_k}$ is increased by $d-1$, \dots, $h_{r_k -2}$ is increased by 1.  Comparing with Equation \eqref{eq:deficit-I},  it follows that the deficit of the original ideal $I$ is decreased, in one step, by a total of $d = d_1 = r_k - s_k$. Repeating the argument we see that the basic double links used in Algorithm \ref{algorithm} reduce the deficit by $\sum d_i$. Since the result of  this algorithm  is numerically Cohen-Macaulay by Proposition \ref{prop:corr-first-alg}, i.e., the deficit for the $h$-vector of $R/J$ is zero, we conclude that
\begin{equation*}
\delta = \sum_{r_i > s_i} (r_i - s_i) = \sum d_i.
\end{equation*}
Applying Lemma \ref{bdl res}, the result about the local cohomology modules follows.

It remains to show that $R/J$ and $R/J'$ have the same $h$-vector.
Indeed, we have seen above that the first basic double link used in Algorithm \ref{algorithm} reduces the deficit  by one in each of  $d$ consecutive degrees beginning with the leftmost possible degree. If needed, the second basic double link similarly reduces the deficit beginning with the then leftmost possible degree. Comparing with the above description for obtaining the $h$-vector of $J'$ from the one of $I$, it follows that the numerical Macaulification $J$ computed by Algorithm \ref{algorithm} has the same $h$-vector as $J'$.  This concludes our argument.
\end{proof}

\begin{example}
If we apply Algorithm \ref{alg2} to the curve in Example \ref{ci33}, we make the following computation to get the $h$-vector of the resulting numerically ACM curve:

\[
\begin{array}{ccccccccccccccccc}
&&&1 & 2 & 3 & 4 & 5 & 6 & 3 & 0 & -3 & -2 & -1 \\
&& 1 & 1 & 1 & 1 & 1 & 1 & 1 & 1 & 1 & 1 \\
& 1 & 1 & 1 & 1 & 1 & 1 & 1 & 1 & 1 & 1 & 1 & 1  \\
1 & 1 & 1 & 1 & 1 & 1 & 1 & 1 & 1 & 1 & 1 & 1 & 1 & 1 \\ \hline
1 & 2 & 3 & 4 & 5 & 6 & 7 & 8 & 9 & 6& 3
\end{array}
\]
which agrees with the one given in Algorithm \ref{algorithm}.

\end{example}

Although the results of Algorithms \ref{algorithm} and \ref{alg2} are numerically equivalent, the numerical Macaulification produced by the former algorithm will typically have fewer minimal generators than the result of the latter algorithm.


\section{The Lazarsfeld-Rao property}

The goal of this section is to understand the role played by the numerically ACM schemes within a fixed even liaison class of codimension two subschemes of $\mathbb P^n$.  It has been shown in a sequence of papers including \cite{LR}, \cite{BBM}, \cite{MDP}, \cite{N-gorliaison} and \cite{nollet2} that any such even liaison class satisfies the so-called {\em Lazarsfeld-Rao property (LR-property)}, recalled below.  In this section we show that within an even liaison class, the subclass of ideals that are numerically ACM itself satisfies a Lazarsfeld-Rao property.  Throughout this section, we consider unmixed codimension two ideals.

We first recall the Lazarsfeld-Rao property, summarized as follows.

\begin{thm}[Lazarsfeld-Rao (LR) Property]
Let $\mathcal L$ be an even liaison class of unmixed codimension two subschemes of $\mathbb P^n$.  Assume that the elements of $\mathcal L$ are not ACM, so that for at least one $i$, $1 \leq i \leq n-2$, we have $M^i := \bigoplus_t H^i(\mathbb P^n, \mathcal I_X(t)) \neq 0$ for all $X \in \mathcal L$.  Then we have a partition $\mathcal L = \mathcal L^0 \cup \mathcal L^1 \cup \mathcal L^2 \cup \dots$, where $\mathcal L^0$ is the set of those $X \in \mathcal L$ for which $M^i$ has the leftmost possible shift, and $\mathcal L^h$ is the set of those elements of $\mathcal L$ for which $M^i$ is shifted $h$ places to the right (this partition does not depend on the choice of $i$).  Furthermore, we have
\begin{itemize}
\item[(a)] If $X_1,X_2 \in \mathcal L^0$ then there is a flat deformation from one to the other through subschemes all in $\mathcal L^0$, which furthermore preserves the Hilbert function.

\item[(b)] Given $X_0 \in \mathcal L^0$ and $X \in \mathcal L^h$ ($h \geq 1$), there exists a sequence of basic double link schemes (see Lemma \ref{bdl res} and Definition \ref{def of bdl}) $X_0,X_1,\dots, X_t$ such that for all $j$, $1 \leq j \leq t$, $X_j$ is a basic double link of $X_{j-1}$, and $X$ is a deformation of $X_t$ through subschemes all in $\mathcal L^h$, all of which furthermore have the same Hilbert function.
\end{itemize}

\end{thm}

We begin by recalling some technical results that will be important in this section.  They were not originally formulated in this generality, but the same proofs work.

\begin{lem}[\cite{BM4}, Proposition 3.1] \label{def lemma 1}
Let $\mathcal L$ be an even liaison class of codimension two subschemes of $\mathbb P^n$, and let $X,Y \in \mathcal L^h$ be elements such that $X$ and $Y$ have the same Hilbert function.  Then there exists an irreducible flat family $\{X_s\}_{s \in S}$ of codimension two subschemes of $\mathbb P^n$ to which both $X$ and $Y$ below.  Moreover, $S$ can be chosen so that for all $s \in S$, $X_s \in \mathcal L^h$ and $X_s$ has the same Hilbert function as $X$ and $Y$.
\end{lem}

We remark that the conclusion that $X_s \in \mathcal L^h$ is a very strong one: it means that all the elements of the flat family are in the same even liaison class, and their modules have the same shift.

\begin{lem}[\cite{BM4} Corollary 3.4] \label{def lemma 2}
If $X,Y$ are codimension two subschemes of $\mathbb P^n$, both in $\mathcal L^h$, and if the general hyperplane section of $X$ has the same Hilbert function as the general hyperplane section of $Y$, then $X$ and $Y$ belong to the same flat family, with the properties described in Lemma \ref{def lemma 1}.
\end{lem}

\begin{lem}[\cite{BM4}, Corollary 3.9 (b)] \label{bm4 result 1}
Let $I$ be an ideal (not necessarily of codimension two).  Then the ideal obtained from $I$ by a basic double link of type $(d,a)$ is numerically equivalent to the ideal obtained from $I$ by a sequence of $a$ basic double links of type $(d,1)$.  In particular, these ideals have the same Hilbert function.
\end{lem}

\begin{lem}[\cite{BM4}, in proof of Lemma 5.2] \label{bm4 result 2}
Let $I$ be an ideal.  Let $a_0 = \min \{   t \ | \ [I]_t \neq 0 \}$.  Let $a,b \geq a_0$ be integers.  Assume that  $b \neq a_0$.   Let $J_1$ be the ideal obtained from $I$ by a sequence of two basic double links, first of type $(a,1)$ and then of type $(b,1)$ (assume that this is possible).  Then it is also possible to do a sequence of basic double links of types $(b-1,1)$ and $(a+1,1)$, resulting in an ideal $J_2$ that is numerically equivalent to $J_1$.
\end{lem}

\begin{lem}[\cite{BM4}, Section 5] \label{main conclusion}
Let $\mathcal L$ be an even liaison class of codimension two subschemes.  Let $X_0$ be an arbitrary minimal element of $\mathcal L$.  Let $X \in \mathcal L$ be an arbitrary element.  Assume that $X \in \mathcal L^h$.  Let $a$ be the initial degree of $I_{X_0}$.  Then associated to $X$ is a uniquely determined sequence of integers $(b,g_2,g_3,\dots,g_r)$ such that
\begin{enumerate}
\item $b \geq 0$;
\item $a < g_2 < \dots < g_r$;
\item $b+r-1 = h$;
\item $X$ is numerically equivalent to the scheme obtained from $X_0$ by a sequence of basic double links of types
\[
(a,b), (g_2,1), (g_3,1), \dots, (g_r,1).
\]
\end{enumerate}
\end{lem}

Lemma \ref{main conclusion} implies that up to numerical equivalence (i.e. in this case, up to flat deformation preserving the cohomology of the ideal sheaf, and in particular preserving the Hilbert function), we only have to consider sequences of basic double links as in (4), i.e. beginning with a minimal element and satisfying conditions (1), (2), (3).

The following observation is a consequence of Lemma \ref{bdl res}.

\begin{lem} \label{num acm preserved}
If $I$ is numerically ACM and $J$ is the result of applying any basic double link to $I$, then $J$ is again numerically ACM.
\end{lem}

The next result gives a class of curves, all of which fail to be  numerically ACM.

\begin{lem} \label{min elts not NACM}
Let $\mathcal L$ be an even liaison class of curves in $\mathbb P^3$.  Then the minimal elements of $\mathcal L$ are not numerically ACM.
\end{lem}

\begin{proof}
Let $C$ be a minimal curve.  We know that all minimal curves have the same minimal free resolution.  In fact, let $M(C)$ denote the Hartshorne-Rao module of $C$.  Combining a result of Rao (\cite{rao} Theorem 2.5) with a result of Martin-Deschamps and Perrin (\cite{MDP} Proposition 4.4), we see that there are free modules $\mathbb F_4$, $\mathbb F_3$, $\mathbb F_2$, $\mathbb F_1$, $\mathbb F_0$ and $\mathbb F$ fitting into the following minimal free resolutions:
\[
\begin{array}{c}
\displaystyle 0 \rightarrow \mathbb F_4 \rightarrow \mathbb F_3 \rightarrow \mathbb F_2 \rightarrow \mathbb F_1 \rightarrow F_0 \rightarrow M(C) \rightarrow 0 \\ \\
\displaystyle 0 \rightarrow \mathbb F_4 \rightarrow \mathbb F_3 \rightarrow \mathbb F \rightarrow I_C \rightarrow 0
\end{array}
\]
(See also \cite{N-gorliaison}.)
Let
\[
\mathbb F_4 = \bigoplus R(-c_i), \ \ \ \mathbb F_3 = \bigoplus R(-b_i) \ \ \ \mathbb F = \bigoplus R(-a_i).
\]
Since $M(C)$ is Cohen-Macaulay, we have $\max \{ c_i \} > \max \{ b_i \}$.  As  $I_C$ is a height two ideal, we have $\max \{ b_i \} > \max \{ a_i \}$. It follows that in the matrix $A$ constructed in Algorithm \ref{algorithm}, the $(1,1)$ entry of the matrix $A$ is $<0$. Thus $C$ is not numerically ACM.  \end{proof}

We believe that Lemma \ref{min elts not NACM} continues to hold for equidimensional codimension two subschemes, but we do not have a proof.

We will prove a Lazarsfeld-Rao structure theorem for the numerically ACM elements of an even liaison class $\mathcal L$.  We illustrate the idea with an example.

\begin{thm} \label{LR for num ACM}
Let $\mathcal L$ be a codimension two even liaison class of subschemes of $\mathbb P^n$.  Let $\mathcal M$ be the subset of $\mathcal L$ consisting of numerically ACM subschemes.  Then $\mathcal M$ also satisfies the LR property.  That is, there are minimal elements of $\mathcal M$, unique up to flat deformation preserving the cohomology of the ideal sheaf (hence also the Hilbert function), and every element of $\mathcal M$ can be produced from a minimal one by a sequence of basic double links followed by a  flat deformation preserving the cohomology. 

Furthermore, the numerical Macaulification of any minimal element in  $\mathcal L$ is a minimal element in  $\mathcal M$. 
\end{thm}

\begin{proof}
In this proof we will heavily use the fact that $\mathcal L$ has the LR property.  We also use the fact that without loss of generality we can assume that our sequence of basic double links is performed with $\deg G = 1$ and $\deg F_i$ {\em strictly increasing. }  We also use the fact that if $C$ is a minimal element with $h$-vector $(1,2,h_2,\dots,h_k)$ (with $h_i$ not necessarily non-negative) and a basic double link of type $(t,1)$ is performed, then the $h$-vector of the resulting curve is obtained by shifting that of $C$ by one, and adding a vector of $(\deg F)$ ones as in the previous example. Finally, recall that $C$ is numerically ACM if and only if its $h$-vector is an $O$-sequence.

First we produce the minimal elements of $\mathcal M$ using Algorithm \ref{alg2}.
Let $C_0$ be a minimal element of $\mathcal L$, and let $\underline{h} = (1,2, h_2,h_3,\dots,h_e)$ be its $h$-vector (recall that  minimal elements of $\mathcal L$ all have the same $h$-vector). If all the minimal elements of $\mathcal L$ are numerically ACM then it follows from Lemma \ref{num acm preserved} and the LR-property for $\mathcal L$ that  $\mathcal M = \mathcal L$, hence $\mathcal M$ has the LR-property.  (In the case of curves in $\mathbb P^3$, we have seen in Lemma \ref{min elts not NACM} that minimal elements are never numerically ACM.)

So assume that the minimal elements are not numerically ACM.  By Lemma \ref{min elts not NACM}, $\underline{h}$ is not an $O$-sequence. Thus there is a least degree $i$ such that $h_i < h_{i+1}$ (possibly $h_i < 0$ as well).  Each application of  the  step $(*)$ of Algorithm \ref{alg2} reduces the deficit by 1 in the leftmost possible degree, and we ultimately obtain an $O$-sequence, hence a numerically ACM subscheme.  This subscheme has deficiency modules (which are not all zero) shifted $\delta$ places to the right of that of $C_0$, where $\delta$ is the number of basic double links applied.

To prove the LR-property for $\mathcal M$, we must show that

\begin{itemize}
\item[(a)] no sequence of fewer than $\delta$ basic double links (always with $\deg G = 1$) achieves a numerically ACM subscheme;

\item[(b)] the (strictly increasing) degrees of the forms obtaining a numerically ACM subscheme in $\delta$ steps is uniquely determined;

\item[(c)] any sequence of basic double links using forms of strictly increasing degrees and ending with a numerically ACM subscheme is numerically equivalent to another sequence whose first $\delta$ steps are the ones from above (but now the remaining steps are not necessarily using forms of strictly increasing degree).
\end{itemize}

Parts (a) and (b) give us that the minimal shift of $\mathcal M$ can be obtained using this procedure, and the elements lie in a flat family since they have the same Hilbert function and cohomology. Part (c) says that any element of $\mathcal M$ can be obtained from a minimal element using basic double linkage. We do not claim that any element of $\mathcal M$ can be obtained from a minimal element using basic double links of strictly increasing degree!

We make two observations at this point.  First, if instead of applying step $(*)$ of Algorithm \ref{alg2}, we make a choice of $F$ of degree $> i+2$, it  results in a situation where no subsequent sequence of basic double links using forms of strictly increasing degree can be numerically ACM, since there will always be a step in the $h$-vector where the deficit is positive.  Second, if we choose $F$ to be of degree $< i+2$, the number of steps remaining to obtain a numerically ACM subscheme does not drop.  These observations together imply (a) and (b).

For (c), suppose that $C$ is a numerically ACM subscheme in $\mathcal L$ that is obtained from a minimal element by some sequence of basic double links using forms of strictly increasing degree (after possibly a number of basic double links using forms of least degree), and then deforming. Let $\delta$ be the deficit of the minimal element.   By the above observations, there are precisely $\delta$ of the basic double links that reduce the deficit by one; the others do not change the deficit.  Consider the first deficit-reducing basic double link. If it is actually the first basic double link, leave it alone and consider the second deficit-reducing basic double link.  If all of the deficit-reducing basic double links come at the beginning of the sequence of basic double links, there is nothing to prove.  Otherwise, suppose that the $i$-th deficit-reducing basic double link is the first one that is not at the $i$-th step, and suppose that it comes at the $(i+k)$-th step. Taking the $(i+k-1)$-st and $(i+k)$-th basic double links, we can apply Lemma \ref{bm4 result 2}, resulting in a new sequence where the first $(i+k-2)$ basic double links are unchanged, and the $(i+k-1)$-st one is deficit-reducing.  Continuing in this way, we produce a new sequence where the first $i$ basic double links are all deficit-reducing.  Then we move to the next deficit-reducing basic double link, and in this way ultimately we produce the desired result of all $\delta$ basic double links coming at the beginning of the sequence (but possibly losing the property of being strictly increasing).
\end{proof}


\section{A class of smooth numerically ACM curves} \label{smooth}

In this section we give a class of smooth numerically ACM space curves that contains Harris's example as a special case.  We will find such an example in the liaison class $\mathbb L_n$ consisting of curves whose Hartshorne-Rao module is $K^n$, i.e. is $n$-dimensional, concentrated in one degree.  The example of Joe Harris is the case $n=1$.  It will first be necessary to recall some facts.  Except as indicated, all of these results can be found in (or deduced from) \cite{BM1}.

Our main tool is a result from \cite{BM1} that is stated in the language of {\em numerical characters}.  The numerical character of a curve was introduced by Gruson and Peskine \cite{GP}.  Rather than give the definition, we give an equivalent formulation from \cite{GM3} that is more useful to us in our setting.

Let $C$ be a curve in $\mathbb P^3$ and let $H$ be a general hyperplane.  Let $\underline{h}$ be the $h$-vector of the hyperplane section $Z = C \cap H$.  Suppose that $\sigma$ is the initial degree of $I_Z$ in the coordinate ring of $H = \mathbb P^2$.  Then $\underline{h}$ reaches a maximum value of $\sigma$ in degree $\sigma-1$ (and possibly beyond).  The numerical character of $C$ is a $\sigma$-tuple of integers defined as follows: First, if $\Delta \underline{h}$ takes a negative value $-k$ in degree $t$ then there are $k$ occurrences of $t$ in the numerical character. These account for all the entries of the numerical character.  However, we list the entries of the numerical character in non-increasing order.  For example, if $Z$ has $h$-vector
\[
(1,2,3,4,5,6,7,8,8,6,5,5,3,2,1)
\]
then $\sigma = 8$ and the numerical character is $(15, 14, 13, 12, 12, 10, 9,9)$. Notice that there is a gap (there is no 11), corresponding to the fact that 5 occurs twice in the $h$-vector.

Since $K^n$ is self-dual, we may view $\mathbb L_n$ as an even liaison class.  The minimal elements of $\mathbb L_n$ form a flat family (thanks to the LR-property) and the general element is smooth.  All minimal elements have degree $2n^2$ and arithmetic genus $\frac{1}{3} (2n-3)(2n-1)(2n+1)$.  The Hartshorne-Rao module of such a curve occurs in degree $2n-2$. The ideal is generated by forms of degree $2n$, and there are $3n+1$ such generators.  The $h$-vector of a minimal curve is
\begin{equation} \label{h-vector of min}
(1,2,3,4,\dots,2n-1,2n,-n).
\end{equation}
Since the Hartshorne-Rao module occurs in degree $2n-2$ and the initial degree of the ideal is $2n$, any minimal curve has maximal rank.

The $h$-vector of the general hyperplane section of such a curve is
\begin{equation} \label{h-vector of hyperplane sect}
(1,2,3,4,\dots,2n-2,2n-1,n);
\end{equation}
in particular, the initial degree is $\sigma = 2n-1$ and there are $n$ generators of that degree.  The numerical character of the curve is
\[
( \underbrace{2n,\dots,2n}_{n},\underbrace{2n-1,\dots,2n-1}_{n-1} ).
\]

Given a curve $C$ and its numerical character $(n_0,n_1,\dots,n_{\sigma-1})$ with $n_0 \geq n_1 \geq \dots \geq n_{\sigma-1}$, we define for any integer $i$
\[
A_i = \# \{ j | n_j = i \}.
\]
One of the main results of \cite{BM1} is the following:

\begin{thm}[\cite{BM1}, Theorem 5.3] \label{BM1 result}
Let $N = (n_0,n_1,\dots,n_{\sigma-1})$ be a sequence of integers without gaps satisfying
\begin{itemize}
\item $n_0 \geq n_1 \geq \dots \geq n_{\sigma-1} \geq \sigma$;

\item $\sigma \geq 2n-1$;

\item $A_\sigma \geq n-1$;

\item $A_{\sigma+1} \geq n$, and if $A_{\sigma+1} = n$ then  $A_t = 0$ for all $t > \sigma+1$.
\end{itemize}
Then there exists a smooth maximal rank curve $C \in \mathbb L_n$ with numerical character $N$.
\end{thm}

Another result proved there, which will be useful to us, is the following. Note that it was shown much earlier by Gruson and Peskine \cite{GP} that the general hyperplane section of an integral curve in $\mathbb P^3$ has a numerical character that is without gaps.  This turns out to be equivalent to the observation by Joe Harris \cite{space curves} that the $h$-vector of the general hyperplane section of an irreducible curve in $\mathbb P^3$ is of {\em decreasing type}.

\begin{thm}[\cite{BM1}, Corollary 3.6, Notation 3.7] \label{BM1 result2}
Let $Y \in \mathbb L_n$ be a smooth maximal rank curve, and let $N(Y) = (n_0,n_1,\dots,n_{\sigma-1})$ be its numerical character. Then the sufficient conditions listed in Theorem \ref{BM1 result} are also necessary.

\end{thm}

We first produce a minimal numerically ACM curve in $\mathbb L_n$ (which will not be smoothable in $\mathbb L_n$).  According to the algorithm in Theorem \ref{LR for num ACM}, since we begin with the $h$-vector (\ref{h-vector of min}) for the minimal curve, we must perform a series of $n$ basic double links, using forms in the ideal of degrees $2n+2, 2n+3,\dots,3n+1$.  The resulting curve, $Y$, is numerically ACM.  Its Hartshorne-Rao module occurs in degree $(2n-2)+n = 3n-2$ and the initial degree of $I_Y$ is $(2n) + n = 3n$.  Thus $Y$ also has maximal rank.

To compute the $h$-vector of the general hyperplane section of $Y$, we begin with the $h$-vector of the general hyperplane section of the minimal curve in $\mathbb L_n$, namely (\ref{h-vector of hyperplane sect}), and perform the same sequence of basic double link calculations using forms of degrees $2n+2, 2n+3,\dots,3n+1$.  We obtain that the general hyperplane section of $Y$ has $h$-vector
\[
(1,2,3,\dots, 3n-2, 3n-1,2n,n)
\]
where the $2n$ occurs in degree $3n-1$.  This translates to a numerical character
\[
(\underbrace{3n+1,\dots,3n+1}_n , \underbrace{3n,\dots,3n}_n , \underbrace{3n-1,\dots,3n-1}_{n-1})
\]
with $\sigma = 3n-1$.  Notice that this does not satisfy the last condition of Theorem \ref{BM1 result}, hence by Theorem \ref{BM1 result2} we cannot find a smooth $Y$ with these numerical properties.   However, if we perform one more basic double link, using a form in $I_Y$ of degree $3n+1$, we obtain a curve $C'$ with Hartshorne-Rao module in degree $3n-1$, initial degree of $I_{C'}$ equal to $3n+1$ (hence $C'$ has maximal rank).  Its general hyperplane section has $h$-vector
\[
(1,2,3,\dots, 3n,2n+1,n),
\]
so the numerical character of $C'$ is
\begin{equation} \label{bdl}
(\underbrace{3n+2,\dots,3n+2}_n, \underbrace{3n+1,\dots,3n+1}_{n+1}, \underbrace{3n,\dots,3n}_{n-1}).
\end{equation}
Now $\sigma = 3n$, and we observe that this numerical character does satisfy the conditions of Theorem \ref{BM1 result}.  The degree of $C'$ can be obtained by adding the entries in the $h$-vector:
\[
\deg C' = \frac{9n^2 + 9n + 2}{2}.
\]

At this point we have a curve $C' \in \mathbb L_n$ that is numerically ACM and also satisfies the conditions of Theorems \ref{BM1 result} and \ref{BM1 result2}.  This is not quite enough to guarantee that $C'$ lies in an irreducible flat family of numerically ACM elements of $\mathbb L_n$, the general one of which is smooth.  A priori, it is conceivable that there are several families, all of whose general {\em hyperplane sections} have the described Hilbert function, but some of which have elements that are numerically ACM and others of which have general elements that are smooth, but none having both.  (This could happen if different shifts of the Hartshorne-Rao module are involved.)  Our next observation is a uniqueness result that eliminates this danger.

Observe that starting with the $h$-vector of the general hyperplane section of the minimal curve, given in (\ref{h-vector of hyperplane sect}), if we perform basic double links using forms of strictly increasing degrees, the {\em only} way to obtain the numerical character (\ref{bdl}) is via the given sequence of $n+1$ basic double links.  Thus in $\mathbb L_n$, there is only one flat family of curves with this numerical character, namely the one containing $C'$, which is numerically ACM.  But the property of being numerically ACM is preserved in the flat family, and by Theorem \ref{BM1 result} this flat family contains a smooth curve, $C$.

We have thus shown:

\begin{thm} \label{family of smooth NACM}
The liaison class $\mathbb L_n$ contains a smooth, maximal rank, numerically ACM curve $C$ of degree $\frac{9n^2 + 9n + 2}{2}$.  $C$ is the  smooth curve of least degree in $\mathbb L_n$ that is numerically ACM, but it does not have least degree simply among the numerically ACM curves in $\mathbb L_n$.
\end{thm}

Notice that when $n=1$, we obtain Harris's curve of degree 10.

\begin{question}
Does Harris's curve have  the smallest possible degree among integral  numerically ACM (but not ACM) curves in $\mathbb P^3$?
\end{question}

\begin{question}
We believe that at least for curves in $\mathbb P^3$, every even liaison class contains smooth numerically ACM elements.  Does every even liaison class of locally Cohen-Macaulay, codimension two subschemes of $\mathbb P^n$ contain smooth numerically ACM elements?
\end{question}

\begin{rem}
We wonder if the smooth numerically ACM curves in $\mathbb L_n$, or perhaps in any even liaison class of curves in $\mathbb P^3$, satisfy some sort of Lazarsfeld-Rao property, similar to what was studied in \cite{nollet} (without regard to the numerically ACM property). Since our main tool here involves only maximal rank curves in $\mathbb L^n$, we do not know the answer to this question.  We remark that thanks to the results in \cite{BM2}, a result similar to Theorem \ref{family of smooth NACM} is probably possible for smooth, maximal rank arithmetically Buchsbaum curves whose Hartshorne-Rao module has diameter two.
\end{rem}

As mentioned in the introduction, smooth curves,  even having the same Hilbert function, can behave very differently depending on whether they are ACM or not.  The following gives an interesting illustration. On an integral ACM curve, there are Gorenstein sets of points with arbitrarily large degree (see \cite{KMMNP}). This in no longer true on non-ACM curves, even if we just ask for zero-dimensional schemes.

\begin{prop} \label{gor subscheme}
Let $C$ be an integral non-ACM  curve in $\mathbb P^n$. Then there is an integer $N$, depending only on the regularity of $I_C$, such that $C$ contains no arithmetically Gorenstein  zero-dimensional scheme of degree $> N$.  In fact, $N = \deg C \cdot \reg I_ C$ has this property.
\end{prop}

\begin{proof}
Let $Z \subset C$ be a zero-dimensional scheme. Let $\mathbb F_\bullet$ be the minimal free resolution of $I_C$, and let $\mathbb G_\bullet$ be the minimal free resolution of $I_Z$. The length of both $\mathbb F_\bullet$ and $\mathbb G_\bullet$ is $n-1$ (the former since $C$ is not ACM).

Let $d = \reg I_ C$.   In particular, all the minimal generators of $I_C$ have degree $\leq d$.  In fact, in the Betti diagram for $\mathbb F_\bullet$, the last non-zero row is the $d$-th one.

 Now assume that $|Z| > d \cdot \deg C$.
Since $C$ is integral, any hypersurface of degree $\leq d$ that contains $Z$ also contains $C$, so  any minimal generator of $I_Z$ that is not in $I_C$ has degree $>d$.  Consequently, the first $d$ rows of the Betti diagram for $I_Z$ are precisely the Betti diagram for $I_C$.  It also follows that the largest twist of $\mathbb G_{n-1}$ is strictly larger than the largest  twist of $\mathbb F_{n-1}$. But the summands of $\mathbb F_{n-1}$ are also summands of $\mathbb G_{n-1}$.  Thus $\mathbb G_{n-1}$ has at least two summands, so $Z$ cannot be arithmetically Gorenstein.
\end{proof}


\section{Examples} \label{example section}

We illustrate our algorithms and results by a few more examples, and we raise some questions that, we believe, deserve further consideration.

\begin{example} \label{illustrate both approaches}
Let $C_1$ be the scheme in $\mathbb P^3$ defined by the cube of the ideal of a general line.  Let $C_2$ and $C_3$ be general complete intersections of types $(1,2)$ and $(4,8)$ respectively.
Let $C = C_1 \cup C_2 \cup C_3$. The Betti diagram for $R/I_C$ is

\begin{verbatim}

        0    1    2    3
-------------------------
 0:     1    -    -    -
 1:     -    -    -    -
 2:     -    -    -    -
 3:     -    -    -    -
 4:     -    -    -    -
 5:     -    -    -    -
 6:     -    -    -    -
 7:     -    4    3    -
 8:     -    4    7    3
 9:     -    -    -    -
10:     -    -    -    -
11:     -    5    4    -
12:     -    4    8    4
13:     -    -    4    3
-------------------------
Tot:    1   17   26   10
\end{verbatim}

\noindent and the $h$-vector of $R/I_C$ is $(1, 2, 3, 4, 5, 6, 7, 8, 5, 1, 4, 4, -1, -6, -3)$.  We illustrate the two algorithms, and why they produce numerically equivalent results.  We begin with the first algorithm.

Collecting our lists of $\{r_i \}$ and $\{ s_i \}$, we obtain
\[
\begin{array}{cccccccccccccccccccccccccccccccccccc}
\{s_i\} & 15 & 15 & 15 & 15 & 14 & 14 & 14 & 14 & 14 & 14 & 14 & 14 & 13 & 13 & 13 & 13 & 10 & 10 & 10 & 10 \\
\{ r_i \} & 16 & 16 & 16 & 15 & 15 & 15 & 15 & 13 & 13 & 13 & 13 & 12 & 12 & 12 & 12 & 12 & 11 & 11 & 11 & 9
\\ \\
 & 10 & 10 & 10 & 9 & 9 & 9 \\
& 9 & 9 & 9 & 8 & 8 & 8 & 8
\end{array}
\]
However, we first remove duplicates and re-align the lists:
\[
\begin{array}{cccccccccccccccccccccccccccccccccccc}
\{s_i\} & 14 & 14 & 14 & 14 & 14 & 14 & 14 & 14 & 10 & 10 & 10 & 10  & 10 & 10 & 10 \\
\{ r_i \} & 16 & 16 & 16 & 12 & 12 & 12 & 12 & 12 & 11 & 11 & 11 & 9 & 8 & 8 & 8 & 8
\end{array}
\]
The negative entries of the main diagonal of the matrix are precisely those integers for which $s_i < r_i$.  We thus immediately see that we will need three basic double links of height 2 and three of height 1, a fact that was not at all evident before removing the duplicates.    More careful analysis shows that in fact the sequence of basic double links consists of types
\[
(11,1), (12,1), (13,1), (19,2), (21,2), \hbox{ and } ( 23,2).
\]
Notice that the  sum of the heights of the basic double links is $1+1+1+2+2+2 = 9$.

As for the second algorithm, we begin with the $h$-vector
\[
(
\begin{array}{ccccccccccccccccccccccc}
1, &  2, &  3, & 4, & 5, & 6, & 7, & 8, &  5, & 1, & 4, & 4, & -1, & -6, & -3.
\end{array}
)
\]
We look for places where the value in one degree is smaller than that in the next degree.  The total deficit is $3 + 3 + 3 = 9$, which is equal to the sum of the heights, i.e. the sum of the absolute values of the negative entries on the main diagonal of the matrix.  In fact, Algorithm \ref{alg2} gives that we must use a sequence of basic double links of type $(d_i,1)$ where $d_i$ takes the values 11, 12, 13, 18, 19, 20, 22, 23, 24.

The resulting curve from the first approach in the same flat family of the even liaison class as the curve resulting from the second approach because the two curves have the same Hilbert function by Proposition \ref{prop:relate-alg}.  Alternatively, we can see this by replacing the basic double links used in the first algorithm by   numerically equivalent basic double links.  First, using Lemma \ref{bm4 result 1} three times we see that we have a sequence of basic double links of type $(d_i,1)$ with $d_i$ taking the values 11, 12, 13, 19, 19, 21, 21, 23, 23.  Applying Lemma \ref{bm4 result 2} six times to this 9-tuple gives sequentially the 9-tuples
\[
\begin{array}{ccccccccccccccccccccc}
11 & 12 & 13 & 19 & 19 & 21 & 21 & 23 & 23 \\
11 & 12 & 13 & 18 & 20 & 21 & 21 & 23 & 23 \\
11 & 12 & 13 & 18 & 20 & 20 & 22 & 23 & 23 \\
11 & 12 & 13 & 18 & 19 & 21 & 22 & 23 & 23 \\
11 & 12 & 13 & 18 & 19 & 21 & 22 & 22 & 24 \\
11 & 12 & 13 & 18 & 19 & 21 & 21 & 23 & 24 \\
11 & 12 & 13 & 18 & 19 & 20 & 22 & 23 & 24
\end{array}
\]
where the last row represents the same sequence of basic double links prescribed (in the end) by the first approach.

We note that in general it seems rather complicated to show directly that the sequences of basic double links used in Algorithms \ref{algorithm} and \ref{alg2} are numerically equivalent.
\end{example}

\begin{example} \label{11 pts}
Let $I \subset k[w,x,y,z]$ be the ideal of a set $Z$ of 11 general points in $\mathbb P^3$.  The Betti diagram for $R/I$ is
\begin{verbatim}
                                    0    1    2    3
                            -------------------------
                             0:     1    -    -    -
                             1:     -    -    -    -
                             2:     -    9   12    3
                             3:     -    -    -    1
                            -------------------------
                            Tot:    1    9   12    4
\end{verbatim}

Even though $I$ has codimension three, we can still apply Algorithm \ref{algorithm} to $I$.  We perform a sequence of basic double links of type $(5,1)$, $(6,1)$, $(7,1)$ and $(9,2)$, obtaining an ideal with Betti diagram
\begin{verbatim}
                                     0    1    2    3
                             -------------------------
                              0:     1    -    -    -
                              1:     -    -    -    -
                              2:     -    -    -    -
                              3:     -    -    -    -
                              4:     -    -    -    -
                              5:     -    -    -    -
                              6:     -    -    -    -
                              7:     -    9   12    3
                              8:     -    4    3    1
                              9:     -    -    1    -
                             -------------------------
                             Tot:    1   13   16    4
\end{verbatim}
and $h$-vector $[1, 2, 3, 4, 5, 6, 7, 8]$.
\end{example}

\begin{example}
Let $Z$ be a set of 11 general points in $\mathbb P^3$ as in Example \ref{11 pts}, and let $I$ be the ideal generated by a general set of four forms in $I_Z$ of degree 4.  $I$ defines $Z$ scheme-theoretically, but is not saturated.    The Betti diagram for $R/I$ is
\begin{verbatim}
                                      0    1    2    3    4
                              ------------------------------
                               0:     1    -    -    -    -
                               1:     -    -    -    -    -
                               2:     -    -    -    -    -
                               3:     -    4    -    -    -
                               4:     -    -    -    -    -
                               5:     -    -    -    -    -
                               6:     -    -    6    -    -
                               7:     -    -    -    -    -
                               8:     -    -    1    -    -
                               9:     -    -    3   16    9
                              ------------------------------
                              Tot:    1    4   10   16    9
\end{verbatim}
Applying Algorithm \ref{algorithm} involves a sequence of seven basic double links, resulting in a Betti diagram
\begin{verbatim}
                                    0    1    2    3    4
                            ------------------------------
                             0:     1    -    -    -    -
                             1:     -    -    -    -    -
                                          ...
                            19:     -    -    -    -    -
                            20:     -    4    -    -    -
                            21:     -    -    -    -    -
                            22:     -    -    -    -    -
                            23:     -    -    6    -    -
                            24:     -    3    -    -    -
                            25:     -    -    1    -    -
                            26:     -    1    3   16    9
                            27:     -    3    7    -    -
                            ------------------------------
                            Tot:    1   11   17   16    9
\end{verbatim}
and an $h$-vector
\[
[1, 2, 3, 4, 5, 6, 7, 8, 9, 10, 11, 12, 13, 14, 15, 16, 17, 18, 19, 20, 21, 18, 15, 12, 9, 9, 9, 9, 9].
\]

\end{example}

\begin{example} \label{2nd half of ex 5.9}
  The Betti diagram of the  curve studied in Example \ref{motivating example} is

                           \begin{verbatim}
                                   0    1    2    3
                           -------------------------
                            0:     1    -    -    -
                            1:     -    -    -    -
                            2:     -    -    -    -
                            3:     -    -    -    -
                            4:     -    -    -    -
                            5:     -    -    -    -
                            6:     -    -    -    -
                            7:     -    -    -    -
                            8:     -    -    -    -
                            9:     -    2    1    -
                           10:     -    -    -    -
                           11:     -    2    3    1
                           12:     -    1    -    -
                           13:     -    1    1    -
                           14:     -    1    -    -
                           15:     -    1    5    2
                           16:     -    2    4    2
                           17:     -    2    2    -
                           -------------------------
                           Tot:    1   12   16    5
\end{verbatim}
One can verify that this gives a matrix with only positive entries in the main diagonal (after removing redundant terms), and so the curve is numerically ACM.
\end{example}

\begin{rem}
Using the methods of this paper, many ACM Hilbert functions of curves can be obtained starting from non-ACM curves. It would be interesting to know which ACM Hilbert functions do {\em not} occur in this way, i.e. which force the curve to be ACM (a trivial example is  if it is a plane curve).  Is it a finite list?  How does the question change if we restrict to smooth or integral curves?  This question has been studied from the point of view of the Hilbert function of the general hyperplane section (see e.g. \cite{gorla}) but this is a different question!
\end{rem}

\begin{rem}
It will be noted that all of our numerically ACM subschemes have codimension two.  It would be interesting to find a construction that produces numerically ACM subschemes of higher codimension.  In this paper we have heavily used methods and results that apply only to codimension two, so it is unlikely that results as complete as those given here will be obtained for higher codimension.  Still, it is an interesting problem.
\end{rem}


\noindent {\bf Acknowledgement.} The first author is very grateful to Joe Harris, not only for the letter in 1983 explaining Joe's example, but even more for the preceding five years in which he patiently explained so much more!


\end{document}